\documentclass{article}
\usepackage{amsmath, amssymb, amsthm} 
\usepackage{enumitem}
\usepackage{hyperref}
\usepackage{esint}
\usepackage{comment}

\theoremstyle{plain}
\newtheorem{theorem}{Theorem}[section]
\newtheorem{lemma}[theorem]{Lemma}
\newtheorem{proposition}[theorem]{Proposition}
\newtheorem{corollary}[theorem]{Corollary}

\theoremstyle{definition}
\newtheorem{definition}[theorem]{Definition}

\newtheorem{remark}[theorem]{Remark}

\newtheorem*{remarkn}{Remark}

\newcommand{\C}{\mathbb{C}}
\newcommand{\R}{\mathbb{R}}
\newcommand{\p}{\partial}
\newcommand{\ddb}{\p\bar{\p}}
\newcommand{\gc}{g_{\beta}}

\DeclareMathOperator{\Ric}{Ric}
\DeclareMathOperator{\Lip}{Lip}

\newcommand{\Address}{{
		\bigskip
		\begin{flushright}
			\textsc{Loughborough University, Department of Mathematics\\
				Schofield Building, LE11 3TU, United Kingdom} \\
			\small{\href{mailto:m.de-borbon@lboro.ac.uk}{\nolinkurl{m.de-borbon@lboro.ac.uk}}} \\
			\vspace{3mm}
			\textsc{Aarhus University, Department of Mathematics \\
			Ny Munkegade 118, 8000 Aarhus, Denmark} \\
			\small{\href{mailto:dmitri.panov@kcl.ac.uk}{\nolinkurl{c.spotti@math.au.dk}}}
		\end{flushright}
}}

\title{SNC K\"ahler--Einstein metrics and RCD spaces}
\author{Martin de Borbon and Cristiano Spotti}
\date{}

\begin{document}

\maketitle

\begin{abstract}
    We show that Kähler--Einstein metrics with cone singularities along simple normal crossing (SNC) divisors define RCD spaces, both in the compact setting and in certain non-compact cases, thereby producing many examples of Einstein RCD spaces. In particular, we show the existence of smooth non-compact $4$-manifolds carrying ALE Ricci-flat RCD$(0,4)$ metrics with \textit{any} space form $S^3/\Gamma$ as the link of the tangent cone at infinity, answering a question raised by D. Semola.
    Our proofs rely on the characterization of RCD spaces in the almost-smooth setting due to S. Honda and Honda--Sun.
\end{abstract}

\section{Introduction}

Several recent studies on Ricci limit spaces have highlighted the relevance of the RCD condition in this context (e.g., \cite{AmbrosioGigliSavare2014}). In the K\"ahler setting, this perspective has been explored, for instance, in \cite{Szekelyhidi2024, GuoSong2025, FuGuoSong2025, ChenChiuHallgrenSzekelyhidiToTong2025}. Under the RCD condition, one can extend Cheeger--Colding's limit theory, which has played a key role both in establishing the existence of K\"ahler--Einstein metrics on Fano manifolds \cite{ChenDonaldsonSun2015III} and in the transcendental construction of their K-moduli compactification (see, e.g., \cite{DonaldsonSun2014I, SpottiSunYao2016, LiWangXu2019}).

In this note, we show that a large class of K\"ahler--Einstein (KE) metrics on complex manifolds $X$ with cone singularities of angle $< 2\pi$ along a simple normal crossing (SNC) divisor $D$ --- a somewhat generic condition on the divisor --- indeed give rise to RCD spaces. Our first main result is the following.

\begin{theorem}\label{KEcpt} 
A conical K\"ahler--Einstein metric on a compact simple normal crossing pair $\big(X^n,\sum_i(1-\beta_i)D_i\big)$ with $\beta_i \in (0,1)$ defines an $RCD(\lambda, 2n)$ space.
\end{theorem}

Concrete examples of such KE spaces for any  values of the scalar curvature are given by metrics singular along SNC hyperplane arrangements in the projective space $\mathbb{CP}^n$. In particular, this gives rise to non-trivial (i.e., not with constant curvature) examples of moduli spaces of Einstein RCD spaces of arbitrarily large dimension. 

Our proof of Theorem \ref{KEcpt} is based on the characterization of RCD spaces for compact quasi-smooth spaces by Honda \cite{honda}. We remark that here we are interested in the ``geometric'' regularity aspects of these singular K\"ahler metrics rather than on their general existence theory (e.g., \cite{guenanciapaun}). 

Having established the RCD property for such KE SNC pairs, one can study the Gromov-Hausdorff limits of the corresponding metrics. In particular, the RCD structure allows one to analyse these limits via tangent cones and rescaled limits.
This provides a natural framework for investigating (partial) compactifications of the corresponding moduli spaces, in analogy with degenerations of smooth K\"ahler-Einstein metrics.

\begin{remarkn}
    When the divisor $D$ is smooth, the K\"ahler-Einstein metrics admit polyhomogeneous expansions \cite{jmr}. In particular, these metrics are examples of stratified spaces and the RCD property has been proved by \cite{BertrandKettererMondelloRichard2021}.
\end{remarkn}

To discuss the non-compact situation, we will use instead the local characterization of RCD spaces by Honda--Sun \cite{hondasun}, in the case of Calabi--Yau ALE spaces conically singular along compact SNC divisors. In particular, in combination with previous results on the existence of such ALE Calabi--Yau \cite{accy}, we obtain the following theorem, answering a question by Daniele Semola \cite{semola}. 

\begin{theorem}\label{ALE4d} For \emph{any} Riemannian spherical space form $S^3/\Gamma$  there are ALE Ricci-flat $RCD(0,4)$ spaces homeomorphic to a smooth manifold, whose metric tangent cone at infinity has such $S^3/\Gamma$ as link. 
\end{theorem}

Such ALE spaces have vanishing ADM-mass and fast decay to the asymptotic cone (see \cite{HeinLeBrun2016}). A famous open problem regarding Ricci-flat ALE spaces asks if, up to orientation reversing, all \textit{smooth} ALE Ricci-flat metrics in dimension four must come from  Kronheimer's hyperk\"ahler examples, or their free quotients \cite{Suvaina2012}. In particular, their flat asymptotic cone at infinity would be topologically constraint.  Theorem \ref{ALE4d} is, of course, not providing any answer to the problem, but it shows that in the larger quasi-smooth RCD category there are counterexamples in which the underlying spaces are still smooth manifolds and the ALE metrics have zero mass.\vspace{.5em}

\noindent \textbf{Organization.} The structure of the paper is as follows. In Section $2$, after recalling the basic definitions, we  discuss the compact case, leading to the proof of main Theorem \ref{KEcpt}. Section $3$ is devoted to the non-compact ALE case, finishing with the proof of main Theorem \ref{ALE4d}.\vspace{.5em}

\noindent \textbf{Acknowledgements.}
We would like to thank Song Sun for answering our questions on his paper with Shouhei Honda \cite{hondasun}. 
We would also like to thank the referee for clarifying the Sobolev to Lipschitz property.
C.S.  thanks the Villum foundation for supporting with the grant Villum YIP+ 00053062.

\section{Compact case}

\subsection{Conical K\"ahler metrics: basic definitions}
Let \(X\) be a compact K\"ahler manifold and let \(D \subset X\) be a simple normal crossing divisor with irreducible decomposition \(D = \sum D_i\). For each irreducible component \(D_i\) of \(D\), let \(\beta_i\) be a real number in the open interval \((0,1)\). 

\begin{definition}\label{def:conicKahler}
    A \textit{conical K\"ahler metric} on the pair \((X, \sum (1-\beta_i) D_i)\) is a smooth K\"ahler metric \(g\) on \(X^{\circ}:= X \setminus D\) satisfying the following condition:
    
    For every point \(p \in D\) there exist complex coordinates \((z_1, \ldots, z_n)\) centred at \(p\) such that \(D = \{z_1 \cdot \ldots \cdot z_k = 0\}\) for some \(1 \leq k \leq n\) with \(D_{i_j} = \{z_j = 0\}\) for \(1 \leq j \leq k\) and 
    \begin{equation}
        C^{-1} \gc \leq g \leq C \gc
    \end{equation}
    for some \(C> 1\), where
    \begin{equation}\label{eq:gcone}
        \gc = \sum_{j=1}^k |z_j|^{2\beta_{i_j}-2} |dz_j|^2 \,+\, \sum_{j=k+1}^n |dz_j|^2 \,.
    \end{equation}
\end{definition}

In the next subsections, we show that certain geometric and analytic  properties as required by Honda's characterization  hold for these metrics. \vspace{.5em}

\noindent \textbf{The local model metric.}
Let \(\beta\) be a positive real number. We denote by \(C(2\pi\beta)\) the \(2\)-cone over a circle of length \(2\pi\beta\) with metric \(dr^2 + \beta^2 r^2 d\theta^2\),
where \((r, \theta)\) are polar coordinates  on \(\R^2\). Letting \(z = ({r}/\beta)^{1/\beta} e^{i\theta}\), the metric  is given by \(|z|^{2\beta-2}|dz|^2\).
We also write \(\C_{\beta}\) for the \(2\)-cone \(C(2\pi\beta)\) when we want to emphasize its complex structure. 

The model cone metric \(\gc\) in Equation \eqref{eq:gcone} is the flat K\"ahler metric on \(\C^n\) obtained as the product 
\[
\C_{\beta_{i_1}} \times \ldots \times \C_{\beta_{i_k}} \times \C^{n-k} \,.
\]
In polar coordinates,
\begin{equation}\label{eq:polar}
    \gc = \sum_{j=1}^k \left( dr_j^2 + \beta_{i_j}^2 r_j^2 d\theta_j^2 \right) \,+\, \sum_{j=k+1}^{n} |dz_j|^2 \,.
\end{equation}
In particular, we note that in polar coordinates the model metric \(\gc\) is uniformly equivalent to the Eucidean metric, with
\[
(\min_j \beta_{i_j})^2 \cdot g_{\mathrm{Euc}} \leq \gc \leq g_{\mathrm{Euc}} \,.
\]

\subsection{Almost smooth metric measure space structure}

Let \(g\) be a conical K\"ahler metric as in Definition \ref{def:conicKahler}. 

\begin{lemma}\label{lem:completion}
    Let \(d_g^{\circ}\) be the metric \(X^{\circ} \times X^{\circ} \to \R_{\geq 0}\), induced by the smooth Riemannian metric \(g\) on \(X^{\circ} = X \setminus D\). Then the following holds:
    \begin{enumerate}[label=\textup{(\roman*)}]
        \item the metric \(d_g^{\circ}\) extends over \(D\) as a continuous function \(d_g: X \times X \to \R_{\geq 0}\)\,;
        \item the function \(d_g\) defines a metric on \(X\);
        \item the metric topology of \((X, d_g)\) agrees with the manifold topology of \(X\);
        \item \((X, d_g)\) is the metric completion of \((X^{\circ}, d_g^{\circ})\) and also a length space.
    \end{enumerate}
\end{lemma}

\begin{proof}
    (i) To prove that \(d_g^{\circ}\) extends continuously, 
    it suffices to show that for every point \(x \in D\) and every \(\epsilon >0\), there exists an open set \(U_{x, \epsilon} \subset X\) such that \(x \in U_{x, \epsilon}\) and the diameter of \(U_{x, \epsilon} \cap X^{\circ}\) with respect to \(d^{\circ}_g\) is \(< \epsilon\). The existence of \(U_{x, \epsilon}\) follows from the local inequality \(g \leq C \gc\) and the fact that the cone metric \(\gc\) induces the usual topology of \(\C^n\). 

    (ii) Since \(d_g^{\circ}\) is a metric, it follows by continuity that \(d_g\) is non-negative, symmetric, and satisfies the triangle inequality. It only remains to show that if \(x, y \in X\) are distinct, then \(d_g(x,y)>0\). If \(x,y \in X^{\circ}\) then \(d_g(x,y) = d_g^{\circ}(x,y) > 0\). Thus, we can assume that at least one of the points is not in \(X^{\circ}\), say \(x \in D\). Then there is an open neighbourhood \(U\) of \(x\) such that \(g \geq C^{-1} \gc\) on \(U\). It is then clear that if \(y \in U\) then \(d_g(x,y) > 0\) and there is some \(\epsilon_0>0\) such that \(d_g(x,y) \geq \epsilon_0\) for all \(y \notin U\).  

    (iii) Since \(d_g\) is continuous, open metric balls \(B(x, r)\) with respect to \(d_g\) are open subsets of \(X\). Conversely, if \(U \subset X\) is open and \(x \in U\) then, since \(X \setminus U\) is compact, the function \(y \mapsto d_g(x, y)\) for \(y \in X \setminus U\) attains a minimum \(m\). By (ii), \(m > 0\). It follows that \(B(x, r) \subset U\) for any \(r<m\), showing that \(U\) is open in the metric topology of \((X, d_g)\).

    (iv) Since \(X\) compact, by item (iii) the metric space \((X, d_g)\) is complete. Since \(X^{\circ}\) is dense in \(X\), this shows that \((X, d_g)\) is the metric completion of \((X^{\circ}, d_g^{\circ})\). Finally, we note that the metric completion of a length space is also a length space.
\end{proof}

The volume form of the model metric \(\gc\) is locally integrable near \(D\).
Let \(\mu_g\) denote the obvious extension of the Riemannian measure on \(X^{\circ}\) given by
\[
\mu_g(A) = \int_{A \cap X^{\circ}} dV_g \,.
\]

\begin{lemma}\label{lem:almostsmooth}
    The metric measure space \((X, d_g, \mu_g)\) is almost-smooth.
\end{lemma}

\begin{proof}
    According to \cite[Definition 4.13]{hondasun} we need to check that the singular set of \(g\) has zero \(2\)-capacity, i.e., that there is a sequence of cut-off functions \(\varphi_k\) with \(\varphi_k|_D = 1\) for all \(k\) and \(\varphi_k = 0\) on each fixed compact subset of \(X^{\circ}\) for \(k \gg 1\) such that \(\lim_{k \to \infty} \int_X |\nabla \varphi_k|^2 = 0\). The explicit construction of these cut-off functions is provided in \cite[p. 261]{mondello}.
\end{proof}

\subsection{Sobolev to Lipschitz and \(L^2\)-strong compactness}

Next, we show that the metric measure space \((X, d_g, \mu_g)\)
satisfies the Sobolev to Lipschitz \cite[Definition 2.2]{honda} and \(L^2\)-strong compactness \cite[Definition 3.5]{honda} conditions. 


\begin{lemma}\label{lem:SL}
    The metric measure space \((X, d_g, \mu_g)\) satisfies the Sobolev to Lipschitz property.
\end{lemma}

\begin{proof}
    According to \cite[Definition 2.2]{honda} we need to show that if 
    \[
    f \in H^{1,2}(X, d_g, \mu_g)
    \] 
    satisfies \(|\nabla f| \leq 1\) \(\mu_g\)-a.e. in \(X\) then \(f\) has a \(1\)-Lipschitz representative. By Lemma \ref{lem:completion} we know that the distance function on \(X\) is the continuous extension of the Riemannian distance \((X^{\circ}, d^{\circ}_g)\). The result then follows arguing as in \cite[Proposition 9]{Szekelyhidi2024}.
\end{proof}

\begin{lemma}\label{lem:strongcompact}
    The metric measure space \((X, d_g, \mu_g)\) satisfies the \(L^2\)-strong compactness condition.
\end{lemma}

\begin{proof}
    According to \cite[Definition 3.5]{honda} we have to show that the inclusion of \(H^{1,2}(X, d_g, \mu_g)\) in \(L^2(X, \mu_g)\) is compact. Let \((f_j)_{j=1}^{\infty}\) be a sequence of functions in \(H^{1,2}(X, d_g, \mu_g)\) with \(\|f_j\|_{H^{1,2}(X)} \leq 1\). Take a finite cover of \(X\) by open sets \(U_1, \ldots, U_N\) such that on each \(U_i\) the metric \(g\) is uniformly equivalent to either the Euclidean metric or the model cone metric \(\gc\). Since \(\gc\) is uniformly equivalent to the Euclidean metric in polar coordinates, the \(L^2\)-strong compactness property holds on the open sets \(U_i\) of the cover. Since the cover is finite, by taking a diagonal subsequence \((f_{j_k})_{k=1}^{\infty}\) and noting that sub-sequential limits agree on intersections,
    we can assume that there is a function \(f \in L^2(X, \mu_g)\) such that \(f_{j_k} \to f\) as \(k \to \infty\) in \(L^2(U_i, \mu_g)\) for all \(1 \leq i \leq N\). Therefore,
    \[
    \|f - f_{j_k}\|_{L^2(X)} \leq \sum_{i=1}^N \|f - f_{j_k}\|_{L^2(U_i)} \to 0
    \]
    as \(k \to \infty\), the subsequence \((f_{j_k})_{k=1}^{\infty}\) converges in \(L^2(X, \mu_g)\), and this finishes the proof of the lemma.
\end{proof}

\subsection{Schauder estimates}

We recall the Schauder estimates from \cite{guosong}. \vspace{.5em}

\noindent \textbf{Local setup.}
Consider \(\C^n\) equipped with the model cone metric
 \begin{equation*}
        \gc = \sum_{i=1}^k |z_i|^{2\beta_{i}-2} |dz_i|^2 \,+\, \sum_{a=k+1}^n |dz_a|^2 \,,
\end{equation*}
where \(\beta_1, \ldots, \beta_k \in (0,1)\).
We denote indices that run over \(1, \ldots, k\) with letters \(i, j\) while we use letters \(a, b\) to denote indices that run over \(k+1, \ldots, n\). 

We denote by \(C^{\alpha, \beta}\) the space of H\"older continuous functions \(f\) that are \(C^{\alpha}\) in polar coordinates \eqref{eq:polar}.
Consider the frame of \((1,0)\)-forms given by
\begin{equation}
    \begin{cases}
        \varepsilon_i = |z_i|^{\beta_i-1} \cdot dz_i &\text{ for }\,\, 1 \leq i \leq k  \,,\\
        \varepsilon_a = dz_a &\text{ for }\,\, k+1 \leq a \leq n \,.
    \end{cases} 
\end{equation}
The reason for considering such frame is that its elements have constant length and are pairwise orthogonal with respect to the model metric \(\gc\).    
Let \(\alpha \in (0,1)\), furthermore we assume that 
\begin{equation}\label{eq:alpharestriction}
    \alpha < \min_i \left(\frac{1}{\beta_i} - 1\right) \,.
\end{equation}

A \((1,0)\)-form \(\eta\) is \(C^{\alpha, \beta}\) if we can write
\[
\eta = \sum_{i=1}^k \eta_i \cdot \varepsilon_i \,+\, \sum_{a=k+1}^n \eta_a \cdot \varepsilon_a \,,
\]
where the coefficients \(\eta_{i}\) and \(\eta_a\) are \(C^{\alpha, \beta}\) functions. Moreover, it is also required that the functions \(\eta_i\) vanish along \(\{z_i = 0\}\) for \(1 \leq i \leq k\).
Similarly, a \((1,1)\)-form \(\eta\) is \(C^{\alpha, \beta}\) if the components of \(\eta\) with respect to the frame composed by the \((1,1)\)-forms \(\varepsilon_i \wedge \bar{\varepsilon}_j\), \(\varepsilon_i \wedge \bar{\varepsilon}_a\), \(\varepsilon_a \wedge \bar{\varepsilon}_i\), and \(\varepsilon_a \wedge\bar{\varepsilon}_b\) are \(C^{\alpha, \beta}\) functions. It is also required that the \(\varepsilon_i \wedge \bar{\varepsilon}_a\) and \(\varepsilon_a \wedge \bar{\varepsilon}_i\) components vanish along \(\{z_i = 0\}\).
A function \(u\) is \(C^{2, \alpha, \beta}\) if \(u\), \(\p u\), and \(\ddb u\) are  \(C^{\alpha, \beta}\).
\vspace{.5em}

\noindent \textbf{Global setup.}
Let \(\omega\) be a conical K\"ahler metric on a compact K\"ahler manifold as in Definition \ref{def:conicKahler}. The metric \(\omega\) is \emph{H\"older continuous} if for every point \(p \in D\) we can find local coordinates where \(D = \{z_1 \cdot \ldots \cdot z_k = 0\}\) such that
\[
\omega = \omega_{\mathrm{cone}} + \eta
\]
where \(\eta\) is a \(C^{\alpha, \beta}\) \((1,1)\)-form that vanishes at \(p\). Equivalently, in a neighbourhood of \(p\) we can write \(\omega = i \ddb \varphi\) for some \(\varphi \in C^{2, \alpha, \beta}\).
The function space \(C^{2, \alpha, \beta}(X)\) is defined in the obvious way by taking an open cover. 

The main result we want to recall is the following.

\begin{theorem}\label{thm:schauder}
    Let \(g\) be a conical K\"ahler metric on the pair \((X, \sum(1-\beta_i)D_i)\). Moreover, assume that \(g\) is \(C^{\alpha, \beta}\) for some \(\alpha\) as in Equation \eqref{eq:alpharestriction}.
    If \(u \in H^{1,2}\) is a weak solution of \(\Delta_g u = f\) with \(f \in C^{\alpha, \beta}\). Then \(u \in C^{2, \alpha, \beta}\).
\end{theorem}

\begin{proof}
    This is part of \cite[Corollary 3.41]{guosong}.
\end{proof}

\begin{corollary}\label{cor:schauder}
    Let \(g\) be a conical K\"ahler metric on the pair \((X, \sum(1-\beta_i)D_i)\). Moreover, assume that \(g\) is H\"older continuous. If \(u \in H^{1,2}\) is a weak solution of \(\Delta_g u = \lambda u\) for some \(\lambda \in \R\), then \(u\) is Lipschitz.
\end{corollary}

\begin{proof}
    By De Giorgi-Nash-Moser elliptic regularity, the function \(u\) is H\"older continuous. By Theorem \ref{thm:schauder}, \(u\) is \(C^{2, \alpha, \beta}\). In particular, the gradient of \(u\) is uniformly bounded, so \(u\) is Lipschitz.
\end{proof}

\begin{remarkn}
    The assumption that the cone angles are \(< 2\pi\) is crucial for Corollary \ref{cor:schauder} to hold. Indeed, if \(\beta> 1\) then the function \(u = r^{1/\beta} \cos(\theta)\) is harmonic on the \(2\)-cone \(C(2\pi\beta)\) but it is not Lipschitz.
\end{remarkn}

\subsection{Proof of Theorem \ref{KEcpt}}

We begin by establishing the following general result.

\begin{theorem}\label{thm:Ricci}
    Let \(g\) be a conical K\"ahler metric on the pair \((X, \sum (1-\beta_i) D_i)\). Assume that \(g\) is H\"older continuous and that \(\Ric(g) \geq \lambda g\) on \(X^{\circ}\). Then \((X, d_g, \mu_g)\) is an \(RCD(\lambda, 2n)\) space.
\end{theorem}

\begin{proof}
    By Lemma \ref{lem:almostsmooth}, the triple \((X, d_g, \mu_g)\) is a \((2n)\)-dimensional almost smooth compact metric measure space. Moreover, we have:
    \begin{enumerate}[label=\textup{(\arabic*)}]
        \item by Lemma \ref{lem:SL} the Sobolev to Lipschitz property holds;
        \item by Lemma \ref{lem:strongcompact} the \(L^2\)-strong compactness condition holds;
        \item by Corollary \ref{cor:schauder} the eigenfunctions of the Laplacian are Lipschitz;
        \item by assumption, \(\Ric(g) \geq \lambda g\) on \(X^{\circ}\).
    \end{enumerate}
    Thus, we can apply \cite[Corollary 3.10]{honda} to conclude that \((X, d_g, \mu_g)\) is an \(RCD( \lambda, 2n)\) space.
\end{proof}

We specify to the K\"ahler-Einstein situation.

\begin{definition}\label{def:conicKE}
    Let \(g\) be a conical K\"ahler metric on the pair \((X, \sum (1-\beta_i) D_i)\). We say that \(g\) is a \emph{conical K\"ahler-Einstein metric} if the following conditions are satisfied:
    \begin{enumerate}[label=\textup{(\roman*)}]
        \item \(g\) is H\"older continuous;
        \item \(\Ric (g) = \lambda g\) on \(X^{\circ}\) for some \(\lambda \in \R\).
    \end{enumerate}
\end{definition}

\begin{remark}
    It seems to us that item (i) in the above definition is indeed redundant  and implied by (ii), as it follows by an adaptation of the technique in \cite{chenwang} for a smooth divisor to the simple normal crossing setting using the Schauder estimates from \cite{guosong}.
\end{remark}

\begin{proof}[Proof of Theorem \ref{KEcpt}]
    According to Definition \ref{def:conicKE}, if \(g\) is a conical K\"ahler-Einstein metric on the pair \((X, \sum (1-\beta_i) D_i)\), then \(g\) satisfies the conditions of Theorem \ref{thm:Ricci}. Therefore \((X, d_g, \mu_g)\) is an \(RCD(\lambda, 2n)\) space.
\end{proof}

\begin{remark} Similar arguments as the ones used in this section should also give that other KE metrics singular along more general divisors are RCD spaces: for instance, the Ricci-flat metrics asymptotic to PK cones on more general (non-SNC) hyperplane arrangements in $\mathbb{CP}^2$ constructed in \cite{deBorbonSpotti2023}.

\end{remark}

\section{Non-compact case}

In this section, we show that certain (Ricci-flat) non-compact spaces with conical singularities along a \textit{compact} SNC divisor are also RCD. We focus our discussion on the ALE case, also because our main goal is proving Theorem \ref{ALE4d}, but the arguments apply verbatim as well in the more general asymptotically conical case (for instance, for the Ricci flat Calabi-Yau spaces of this type constructed in \cite{accy}). 

Compared to the previous section, the main tool we use in our proof is the local characterization of  the RCD condition for (possibly non-compact) quasi-smooth spaces in \cite{hondasun}.

\subsection{ALE conical K\"ahler metrics}

We consider the following class of singular metrics on non-compact spaces.

\begin{definition}
    Let \(X\) be a non-compact complex manifold and let \(D \subset X\) be a compact simple normal crossing divisor with irreducible decomposition \(D = \sum D_i\). For each irreducible component \(D_i\) of \(D\), let \(\beta_i\) be a real number in the open interval \((0,1)\). An ALE conical K\"ahler metric on the pair
    \((X, \sum (1-\beta_i) D_i)\) is a smooth K\"ahler metric \(g\) on \(X^{\circ}= X \setminus D\) satisfying the following properties.

    \begin{enumerate}[label=\textup{(\roman*)}]
    \item For every point \(p \in D\) there exist complex coordinates \((z_1, \ldots, z_n)\) centred at \(p\) such that \(D = \{z_1 \cdot \ldots \cdot z_k = 0\}\) and 
    \begin{equation}
        C^{-1} \gc \leq g \leq C \gc
    \end{equation}
    for some \(C> 1\), where \(\gc\) is the model metric given by Equation \eqref{eq:gcone}.

    \item There exists a compact set $K\subset X$ and a diffeomorphism onto its image $\phi:X\setminus K \longrightarrow \mathbb{C}^{n}/\Gamma$, with $\Gamma$ finite group in $U(n)$ acting freely away from zero such that
$$ \Vert \nabla_0^k \big( (\phi^{-1})^* g-g_0 \big)\Vert_{g_0}=O(\vert z \vert^{-\lambda-k}) \quad \textup{ for all } k \geq 0,$$
where $g_0$ denotes the flat metric, $\lambda >0$ and $\vert z \vert \gg 1$.
\end{enumerate}    
The metric \(g\) is H\"older continuous if locally near \(D\) it admits \(C^{2,\alpha, \beta}\) K\"ahler potentials, as explained in the previously discussed compact case.
\end{definition}

Let \(g\) be an ALE conical K\"ahler metric on the pair \((X, \sum (1-\beta_i) D_i)\). 
In full analogy to the compact case discussed previously, the distance function defined by the Riemannian metric \(g\) on \(X^{\circ}\) extends continuously over \(D\) and the metric measure space $(X, d_g, \mu_g)$ is an almost smooth complete length space. 
We prove some basic results about $(X, d_g, \mu_g)$ that we need in order to apply Honda-Sun's characterization results of RCD spaces in the quasi-smooth case.

\subsection{Volume doubling and Poincar\'e inequality}

\begin{proposition}\label{prop:voldoub} $(X, d_g, \mu_g)$ as above  satisfies (global) volume doubling, i.e., there exists $C>0$, such that $\mu_g(B(x,2r))\leq C\, \mu_g(B(x,r))$ for each $x\in X$.
\end{proposition}

\begin{proof}
We will show that there is \(C_1>0\) such that for all \(x \in X\) and all \(r>0\),
\begin{equation}\label{eq:volgrowth}
   r^{2n}/C_1\leq \mu_g (B(x,r)) \leq C_1 r^{2n} \,.
\end{equation}
The proposition follows from this with constant \(C = C_1^2 \, 2^{2n}\) as
\[
\mu_g(B(x,2r))\leq C_1 \, 2^{2n} r^{2n}\leq C_1^2 2^{2n} \mu_g(B(x,r)) \, .
\]

To show \eqref{eq:volgrowth}, note that the local uniform equivalence together with the ALE condition ensure that for every compact set \(K \subset X\) there is a constant $C_K$ such that \eqref{eq:volgrowth} holds for all \(x \in K\) and all \(r>0\).
Let \(K_0\) be a compact set such that \(g\) is uniformly equivalent to the Euclidean metric \(g_0\) on \(X \setminus K_0\). Then \eqref{eq:volgrowth} clearly holds if \(B(x,r) \subset X \setminus K_0\).
Let \(\rho\) denote the distance to a fixed point \(p \in K_0\).
Take a larger compact set \(K_0 \subset K\) such that whenever \(x \notin K\) and \(B(x, r)\) intersects \( K_0\), one necessarily has \(r \geq r_*\) for some fixed \(r_*>0\), moreover \(B(x, r/2) \subset X \setminus K_0\), and \(\mu_g(K_0) \leq r^{2n}\). It follows from these properties that there is \(C'>0\) such that for all \(x \in X \setminus K\) and all \(r>0\) we have \(r^{2n}/C'\leq \mu_g (B(x,r)) \leq C' r^{2n}\). Equation \eqref{eq:volgrowth} follows from this by taking \(C_1 = \max \{C_K, C'\}\)\,.
\end{proof}

\begin{proposition}\label{prop:poincare}
    \((X, d_g, \mu_g)\) satisfies a local (1,2)-Poincar\'e inequality, i.e., there is \(\Lambda>1\) and a function \(c: (0, +\infty) \to (0, +\infty)\) such that for any \(R>0\) and any geodesic ball \(B \subset X\) of radius \(0 < r \leq R\) we have
    \begin{equation}\label{eq:poincare}
        \fint_B |f - f_B| d\mu_g \leq c(R) \, r \, \left( \fint_{\Lambda B} (\Lip f)^2 d\mu_g \right)^{1/2} \,,
    \end{equation}
    where \(f\) is a Lipschitz function on \(X\) and \(f_B = \fint_B f\) denotes its average.
\end{proposition}

\begin{proof}
    We first show that there exist constants \(r_0>0\) and \(C_0 >0\) such that \eqref{eq:poincare} holds for any ball \(B \subset X\) of radius \(r \leq r_0=R\) with \(c(R) = C_0\) and \(\Lambda = 2\), say. Indeed, we can choose \(r_0>0\) small enough such that if \(\tilde{B} \subset X\) is a ball of radius \(r_0\) then the metric \(g\) is bi-Lipschitz on \(2\tilde{B}\) to the Euclidean metric or \(\gc\) with bi-Lipschitz constant \(2\), say. Since the Poincar\'e inequality holds for these models, we conclude that \eqref{eq:poincare} holds for some uniform constant \(C_0>0\) for all balls \(B \subset X\) of radius \(r \leq r_0\). Next, we claim that given any \(1 < \Lambda \leq 2\) and for any \(R>0\), we can choose a constant \(c(R)\) such that \eqref{eq:poincare} holds for any ball \(B \subset X\) of radius \(r \leq R\). This follows automatically from the volume doubling (Proposition \ref{prop:voldoub}) and a standard covering argument (see \cite[Lemma 5.3.1]{saloff} and \cite[Remark 2.23]{hondasun}). Namely, we can cover \(B\) by (a finite number depending only on \(R\)) of balls \(B_i\) of radius \(\leq r_0\) such that the half balls \(2^{-1}B_i\) are disjoint. Then a chaining argument as in \cite[(5.3.10) and (5.3.11)]{saloff} gives the result. This finishes the proof of the proposition.
\end{proof}

\subsection{Proof of Theorem \ref{ALE4d}}

We begin by proving the following extension of Theorem \ref{thm:Ricci}.

\begin{theorem}\label{thm:Ricci2}
    Let \(g\) be an ALE conical K\"ahler metric on \((X, \sum (1-\beta_i) D_i)\). Moreover, assume that \(g\) is H\"older continuous and that \(\Ric(g) \geq \lambda g\) on \(X^{\circ}\). Then \((X, d_g, \mu_g)\) is an \(RCD(\lambda, 2n)\) space.
\end{theorem}

\begin{proof}
As in the compact case discussed previously, the metric measure space $(X, d_g, \mu_g)$ is an almost smooth complete length space. In order to apply the Honda-Sun uniform local characterization of RCD spaces \cite[Theorem 1.1]{hondasun}, several properties need to hold:

\begin{enumerate}[label=\textup{(\arabic*)}]
\item  Sobolev-to-Lipschitz (SL) property;
 \item semilocal volume doubling (item 1 in \cite[Definition 2.19]{hondasun});
\item  semilocal Poincar\'e inequality (item 2 in \cite[Definition 2.19]{hondasun}); 
\item  Quantative Lipschitz (QL) continuity of harmonic functions (see \cite[Definition 2.28]{hondasun});
\item  lower bound on the Ricci tensor on the smooth part (this is automatically satisfied by our assumptions).
\end{enumerate}

The SL property follows exactly as in the compact case since the metric on \(X\) is obtained as a continuous extension of the Riemannian metric \((X^{\circ}, d^{\circ}_g)\).
By Proposition \ref{prop:voldoub},  $(X, d_g, \mu_g)$ is global volume doubling, which implies (2). By Proposition \ref{prop:poincare}, item (3) is also satisfied. 


To prove the QL property it is enough to do so in balls of radius \(0 < r \leq r_0\) for some fixed \(r_0>0\) to be determined later (see \cite[Remark 2.30]{hondasun}), i.e., we need to show that there is a uniform constant \(C>0\) such that if \(B \subset X\) is a ball of radius \(0< r \leq r_0\) and \(f: B \to \R\) is a harmonic function then
\begin{equation}\label{eq:ql}
    \sup_{\frac{1}{2}B} |\nabla f \,| \leq \frac{C}{r} \fint_B |f| d\mu_g \,.
\end{equation}

Since the metric \(g\) is smooth on \(X \setminus D\) and the ALE property, the standard gradient estimates for harmonic functions in \(\R^{2n}\) imply \eqref{eq:ql} if \(B\) does not intersect \(D\). On the other hand, by our assumption that \(g\) is H\"older continuous near \(D\), we can take \(r_0\) small such that if \(\tilde{B} \subset X\) has radius \(r_0\) and intersects \(D\) then
\(\|g-\gc\|_{C^{\alpha, \beta}(\tilde{B})} < 1/100\), say.
It then follows from the gradient bounds for \(\gc\)-harmonic functions \cite[Lemma 3.2]{guosong} that \eqref{eq:ql} holds. Note that the Guo-Song estimates \cite[Equation 3.5]{guosong} bound the norm of the gradient in terms of the oscillation of the harmonic function; however, by De Giorgi-Nash-Moser the oscillation of a \(\gc\)-harmonic function is controlled by its \(L^1\)-norm.
This finishes the proof of (4).

Since \((X, d_g, \mu_g)\) satisfies the above (1)-(5) properties, by \cite[Theorem 1.1]{hondasun} \((X, d_g, \mu_g)\) is an \(RCD(\lambda, 2n)\) space.
\end{proof}

\begin{remark}
    These arguments would imply the result in the compact case as well, but we have nevertheless decided to give in the previous section the slightly different proof based on the more familiar theory of RCD spaces in the compact situation, hoping it can be beneficial for some readers.
\end{remark}

To conclude  with the proof of our second main Theorem \ref{ALE4d}, first recall the following result.

\begin{theorem} [Theorem 1, \cite{accy}] \label{ALE}
Let $\pi : X \to \mathbb{C}^n / \Gamma$ be a resolution that satisfies 
$K_X = \sum_{j=1}^{N} (\beta_j - 1) E_j$, with $E_j$ SNC and $\beta_j\in (0,1)$. Then for each K\"ahler class in $H^2(X, \mathbb{R})$ there is a unique  (Hölder continuous)  Ricci-flat K\"ahler metric $g_{\mathrm{RF}}$ with cone angle $2\pi \beta_j$ along $E_j$ for $j = 1, \dots, N$ such that
\begin{equation}
\left| \nabla^k_{g_C} \Big( (\pi^{-1})^* g_{\mathrm{RF}} - g_C \Big) \right|_{g_C} = O(r^{-2n-k})
\quad \text{for all } k \ge 0.
\end{equation}
\end{theorem}

Thus:

\begin{proof}[Proof of Theorem \ref{ALE4d}]
    The case $\Gamma \subset SU(2)$ is given by the classical construction of (smooth) hyperk\"ahler ALE metrics by Kronheimer \cite{kronheimer}. The general case follows instead by applying Theorem \ref{thm:Ricci2}. for the metrics given by Theorem \ref{ALE}  on the \textit{minimal} resolution (which, in particular, ensures that the cone angles on the exceptional set are all less than $2\pi$)  of any isolated quotient singularity $\mathbb{C}^2/\Gamma$, with $\Gamma\subset U(2)$, combined with the fact that any finite subgroup of $SO(4)$ acting freely on the sphere is orthogonally conjugated to a subgroup in $U(2)$ (see, e.g.,  Proposition $A.1$ in \cite{zhou2024family}). 
\end{proof}

\bibliographystyle{alpha}
\bibliography{refs}

\Address

\end{document}